 \documentclass[1p]{elsarticle}
 
 \usepackage[latin1]{inputenc}

\journal{Journal of \LaTeX\ Templates}

\usepackage{amssymb}
\usepackage{eucal}
\newcommand{\R}{{\Bbb R}}

\newcommand{\N}{{\Bbb N}}

\usepackage{color}

\usepackage{amsmath}

\newtheorem{thm}{Theorem}
\newtheorem{lem}[thm]{Lemma}
\newtheorem{cor}[thm]{Corollary}
\newtheorem{proposition}[thm]{Proposition}

\newtheorem{definition}[thm]{Definition}
\newtheorem{remark}[thm]{Remark}
\newtheorem{example}[thm]{Example}
\newproof{proof}{Proof}

\begin{document}
\begin{frontmatter}

\title{Nonlinearly determined wavefronts of the Nicholson's diffusive  equation: when small delays are not harmless }

\author[a]{Zuzana Chladn\'a}
\author[b]{Karel Has\'ik}
\author[b]{Jana Kopfov\'a}
\author[b]{Petra N\'ab\v{e}lkov\'a}
\author[c]{\hspace{45mm} Sergei Trofimchuk\corref{mycorrespondingauthor}}
\cortext[mycorrespondingauthor]{\hspace{-7mm} {\it e-mails addresses}: chladna@fmph.uniba.sk (Zuzana Chladn\'a);  karel.hasik@math.slu.cz (Karel Has\'ik);   jana.kopfova@math.slu.cz (Jana Kopfov\'a); petra.nabelkova@math.slu.cz (Petra N\'ab\v{e}lkov\'a); trofimch@inst-mat.utalca.cl (Sergei Trofimchuk, corresponding author) \\}
\address[a]{Department of Applied Mathematics and Statistics, Faculty of Mathematics, Physics and Informatics, Comenius University, Mlynsk\'a dolina, 84248  Bratislava, Slovak Republic}
\address[b]{Mathematical Institute, Silesian University, 746 01 Opava, Czech Republic}
\address[c]{Instituto de Matem\'atica y F\'isica, Universidad de Talca, Casilla 747,
Talca, Chile }

\bigskip

\begin{abstract}
\noindent By proving the existence of non-monotone and non-oscillating wavefronts for the Nicholson's blowflies diffusive equation (the NDE),  we answer  an open question from \cite{GTLMS}.   Surprisingly, these wavefronts can be observed only for sufficiently small delays.  Similarly to the pushed  fronts, obtained waves  are not linearly determined. In contrast,  a broader family of eventually monotone wavefronts for the NDE is indeed determined by properties of the spectra of the linearized equations. Our proofs use essentially several specific characteristics  of the blowflies birth function (its unimodal form and the negativity of its Schwarz derivative, among others).  One of the key auxiliary results of the paper shows that  the Mallet-Paret--Cao--Arino theory of super-exponential solutions for scalar equations can be extended for some classes of second order delay differential  equations.  For the new  type  of non-monotone waves to the NDE, our numerical simulations also confirm their stability properties  established by Mei {\it et al}.  \end{abstract}

\end{frontmatter}

\newpage

\section{Introduction and  main results} \label{intro} 
Nicholson's blowflies delay differential equation 
\begin{equation}\label{NB}
u'(t)=-\delta u(t) + p u(t-\tau)e^{-au(t-\tau)}, \quad u \geq 0, 
\end{equation}
was introduced in 1980  by Gurney, Blythe and  Nisbet \cite{GBN} to 
provide a better description of  the evolution of the population $u(t)$ of  mature adults  of the Australian sheep-blowflies ({\it  Lucilia cuprina}) observed 
in a series of highly careful laboratory experiments realized by A. J. Nicholson \cite{Mur}. The positive parameters $\delta, p, a, \tau$ are the model's specific constants and  simple scaling  of variables 
allows us to assume  that $\delta= a = 1$ without any restriction of generality.  
Equation (\ref{NB}) was introduced as a more elaborate alternative to the delayed logistic equation: in difference with (\ref{NB}), the latter was unable to explain some irregular 
oscillations of $u(t)$ observed  in the collected experimental  data. In a short time, it became clear that Nicholson's blowflies equation represents  a fascinating and  non-trivial  object of investigation from the dynamical point of view. This fact attracted the interest of numerous researchers over the decades, cf.  \cite{BBI,AMC,LTT,Mur,Smith,SOY}. Moreover, following the same logic as in the case of  the delayed logistic equation (cf. \cite{BY}),  in 1996 Yang and So \cite{SOY} introduced the following diffusive version of   (\ref{NB}):
\begin{equation}\label{NBD}
\partial_tu(t,x)=\partial_{xx}u(t,x) - u(t,x) + pu(t-\tau,x)e^{-u(t-\tau,x)} , \quad x \in \mathbb R.
\end{equation}
The positive semi-wavefronts  $u(t,x) = \phi(x +ct)$, \ $\phi(-\infty)=0,$\ $\liminf_{t\to +\infty}\phi(t) >0$, are the fundamental transitory regimes in the dynamics generated by the diffusive Nicholson's equation.  The existence, uniqueness,  oscillation/monotonicity and stability properties of these waves were studied, among many other works,   in \cite{Chern,fhw,AMC,GTLMS,LM,LLLM,SZ,ST,TZ,SLW} and  the non-local version of (\ref{NBD})  was considered, among many other articles, in \cite{G,GR,GSW,GZ,KO,LW,GL,SWZ,YCW,ZY}.  

In this paper, we revisit the topic of  possible shapes for wavefronts $u(t,x) = \phi(x +ct)$,  $\phi(-\infty)=0,$  $\phi(+\infty)=\ln p,$  of equation (\ref{NBD}).  Our first main result, Theorem \ref{T1},  shows that contrarily to the tacitly accepted 
hypothesis, cf. \cite{Chern,LLLM},  of the determinacy of  the shape of  $\phi(t)$ by the spectra of linearized equations at the equilibria  (as it happens, for instance,  in the delayed or nonlocal KPP-Fisher equations \cite{ADN,FZ}), equation (\ref{NBD}) can have  wavefronts  which are neither monotone nor oscillating even if  the linearization of the profile equation at the positive equilibrium has negative eigenvalues.  This implies that  a non-monotone wave with unusually  high leading edge (see Fig. 1) can appear in (\ref{NBD}) even if  the associated linearized equations predict  the existence of exclusively monotone waves. Surprisingly enough, this strange type of wavefront's behavior can occur for  arbitrarily small delays $\tau$, to some extent contradicting the folklore principle {\sl "Small delays are harmless"} of the theory of delay differential equations \cite{Smith}.  The mechanism behind this loss of monotonicity of wavefronts is precisely the same one which causes the  {\it "linear determinacy principle"} \cite{LLW} to fail   for the  monostable  population models possessing the weak Allee effect (leading to the appearance of  {\it pushed or non-linearly determined} waves).
\begin{figure}[h] \label{FF2}
\centering \fbox{\includegraphics[height=7cm, width=7cm]{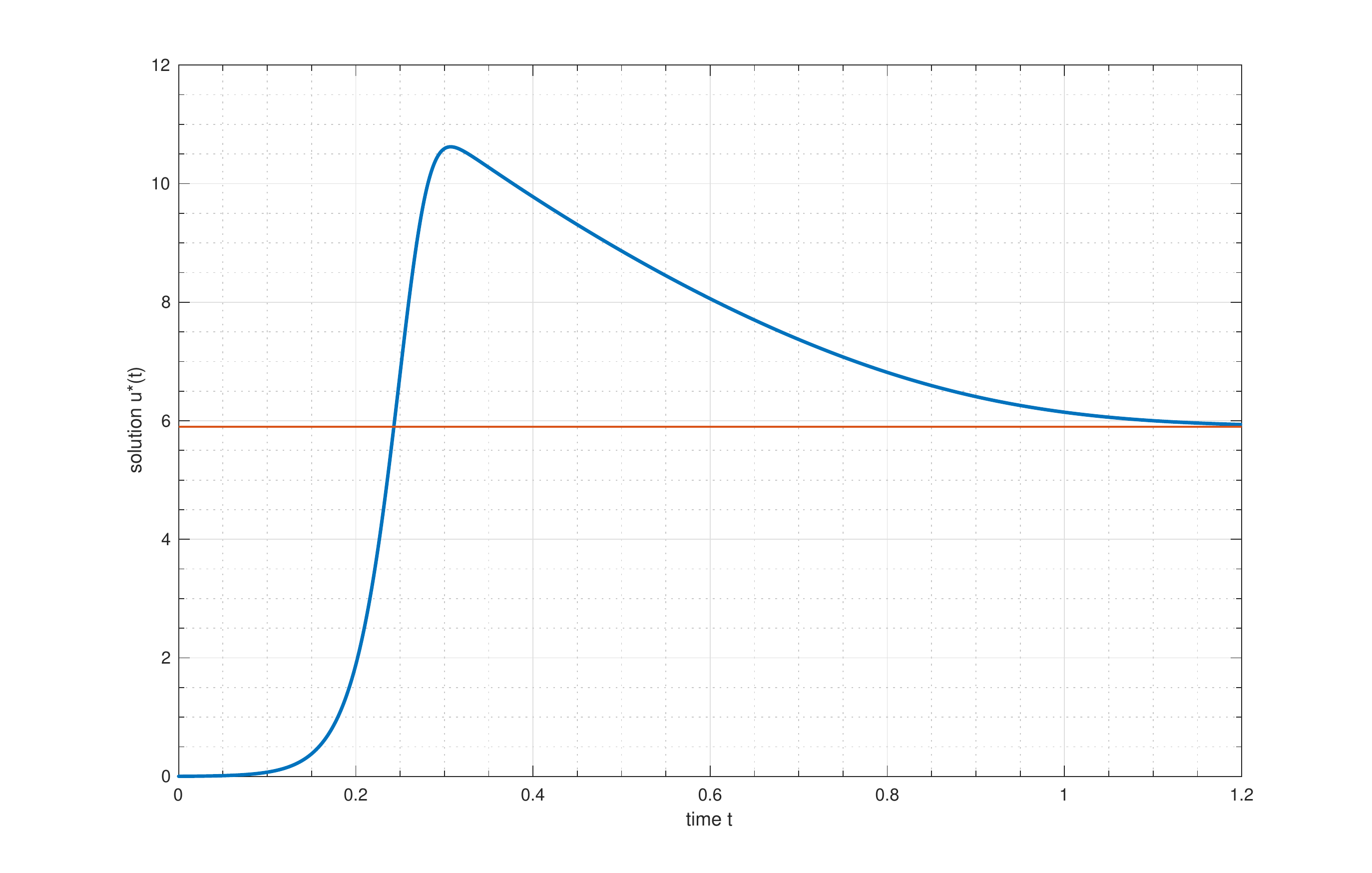}}
\caption{\hspace{0cm} Non-monotone non-oscillating wavefront for equation (\ref{NBD}) with $\tau=0.07, \ p = 365$ and $c= +\infty$.} 
\end{figure}
In order to state our first result, we introduce some notation. In the sequel, $\Gamma(z,s)$ will denote the lower incomplete gamma function, $\chi(z)$ will denote the characteristic function of equation (\ref{NB}) with $\delta =1$ linearized  at $u=0$:
\begin{equation}\label{zeta}
\Gamma(z,s)= \int_0^zt^{s-1}e^{-t}dt, \quad \chi(z) = z+1 -pe^{-z\tau}.
\end{equation}
It is easy to see that for  each $p >1$  equation $\chi(z)=0$ has exactly one positive root $\mu$. Set $\bar q_2 := -pe^{-2\mu\tau}/\chi(2\mu)<0$, $m:= \mu^{-1}$ and 
$$
\zeta := (1+\bar q_2)e^{-\tau} +pm\left(\Gamma(1,m+1) -\Gamma(e^{-\mu\tau},m+1)  + \bar q_2\Gamma(1,m+2) -\bar q_2\Gamma(e^{-\mu\tau},m+2)\right).
$$
\begin{thm} \label{T1}  Let $p, \tau$ be such that $p \in \frak{I}:= (e^2, \exp(1+\exp(-1-\tau)/\tau))$ and 
$\zeta>\ln p$. Then 
there exists  $\hat c(\tau,p)>0$ such that for each $c \geq \hat c(\tau,p)$ equation (\ref{NBD}) has 
a  positive wavefront propagating with the speed $c$ and whose profile $\phi_c(t)$ is eventually monotone at $\pm \infty$ and is non-monotone on $\R$. In fact, $|\phi_c(\cdot)|_\infty \geq \zeta >\ln p$. 
\end{thm}
\begin{example} \label{Ex1} Consider `small' delay  $\tau =0.07$ and take  $p =365 \in \frak{I}$. In such a case, formal  linear analysis   predicts the existence of a unique monotone wavefront connecting $0$ and $\ln p$.   
However,  it is easy to find  that   $\mu =  33.64 \dots$, $\bar q_2 =-0.05\dots$ and $\zeta=  6.46\dots $ $ > $ $\ln p = 5.89\dots$. Therefore the conclusions of Theorem \ref{T1} hold for equation (\ref{NBD}) with 
 $\tau =0.07$, $p =365$. Figure 1 shows a very accurate approximation of the scaled profiles $\phi_c(ct)$ of the corresponding non-monotone and non-oscillating wavefronts (nm-waves, for short) in the limit case $c= +\infty$.  This picture perfectly agrees with the considerations of Remark \ref{RONE} where we present a numerical solution of (\ref{NBD}) with $\tau =0.07$,  $p =365$ converging to  a wavefront propagating with the large  but finite  speed $c\approx 50$. \end{example}
Theorem \ref{T1} and Example \ref{Ex1} answer in positive an open question raised in \cite[p. 53]{GTLMS} about the existence of an eventually monotone and non-monotone front of 
(\ref{NBD}) for some $p > e^2$. In fact, there are many such wavefronts: on Fig. 2, we present the set of parameters $(\tau,\ln \ln p)$ satisfying all requirements  of Theorem \ref{T1}.
Formally, Theorem \ref{T1} works  for sufficiently fast wavefronts: since the minimal speed of propagation $c_*=7.89\dots$  in Example \ref{Ex1} is relatively large (in the sense that 
$c_*^{-2}=0.016\dots \ll \tau$), one might expect that even the minimal wavefront in Example \ref{Ex1} has non-monotone and non-oscillating profile. 
Our numerical simulations confirm this informal conclusion:  on Fig. 3,  we present three consecutive  positions (at the times $t=1, 3, 5$) of solution for equation (\ref{NBD}) with  $\tau =0.07$, $p =365$ and with the Heaviside step function as the initial datum. This numerical solution was obtained by applying the second order central difference schemes for the space derivative. The resulting transformed system of ordinary delayed equations has been solved by Matlab built-in  function dde23.  We take $x\in [-500,500]$ and $t \in [0,5]$. 
This numerical result  complements  the discussion in Section 7 of \cite{Chern}.

\begin{figure}[h] \label{F2}
\centering \fbox{\includegraphics[width=7.5cm]{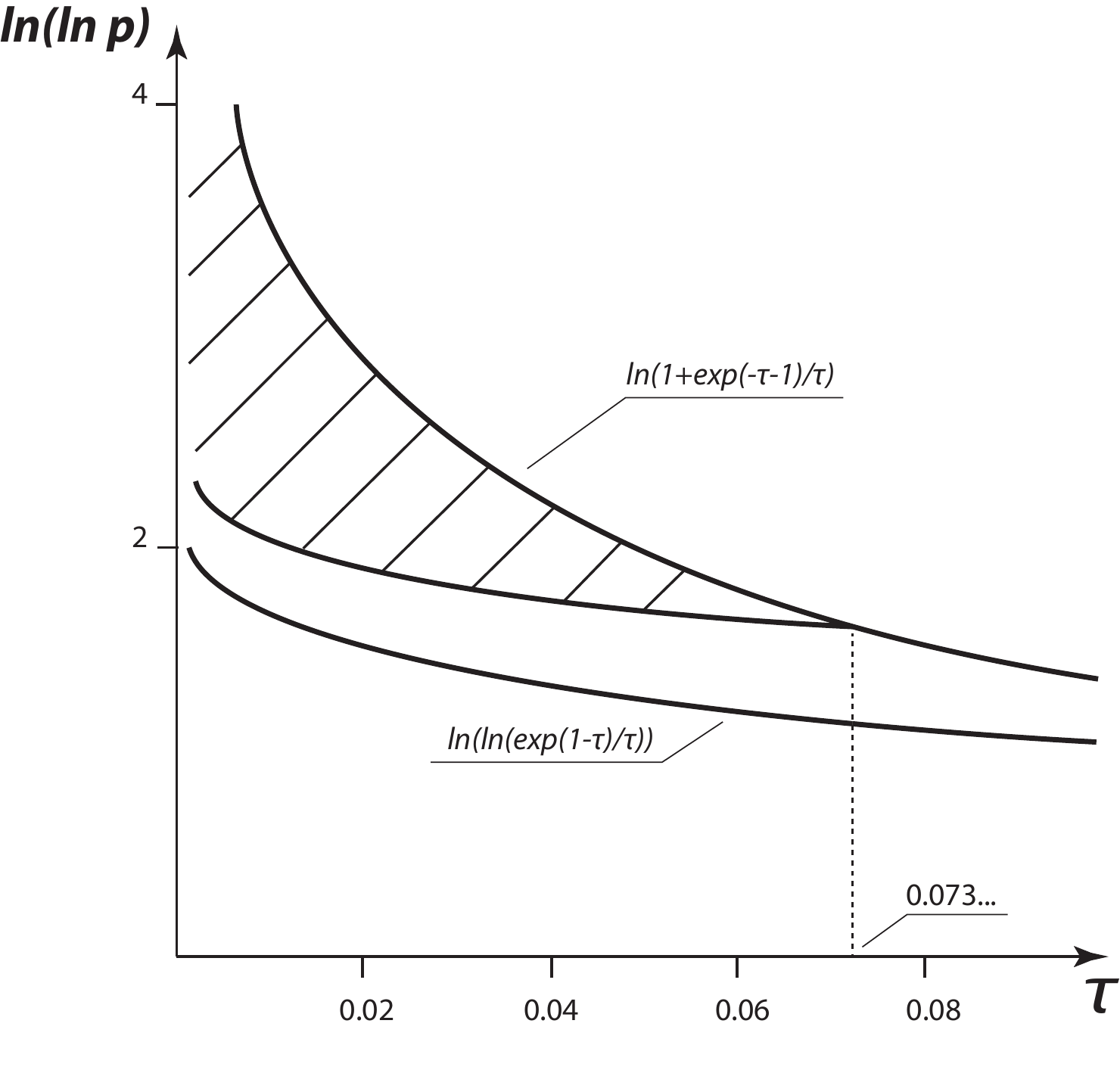}}
\caption{\hspace{0cm} Dashed domain corresponds to parameters $(\tau, p)$ satisfying the assumptions of Theorem \ref{T1}.}
\end{figure}

\begin{figure}[h] \label{F2a}
\centering \fbox{\includegraphics[width=13cm]{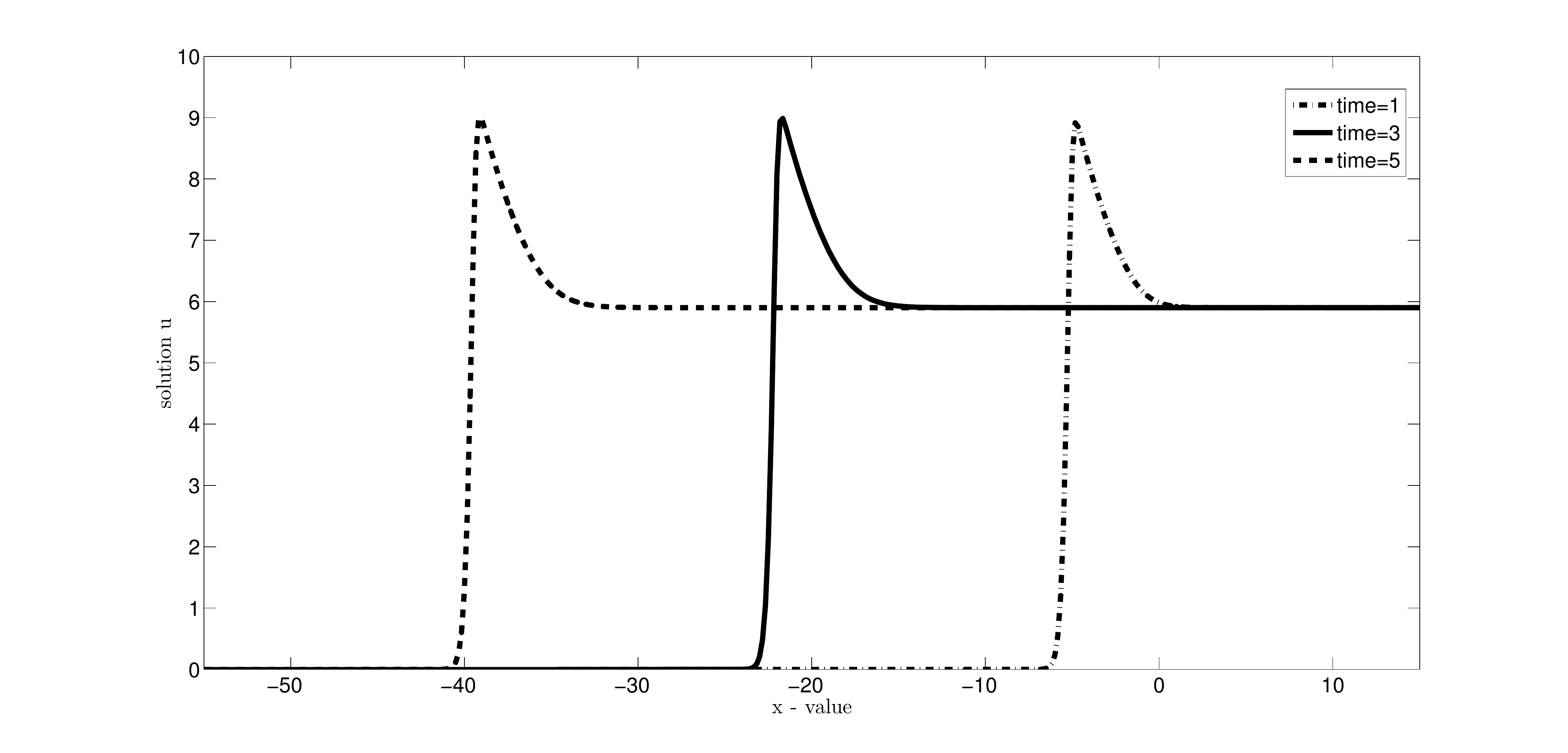}}
\caption{\hspace{0cm} Numerical approximations of the minimal wavefront  for equation  (\ref{NBD}) with  $\tau=0.07$, $p=365$.}
\end{figure}

The Nicholson's blowflies diffusive equation together with the food-limited diffusive equations \cite{HKarxiv,TPT} seem to be the first scalar models  coming from applications where untypical behavior in the form  of non-monotone non-oscillating wavefronts is established analytically and also observed in  numerical experiments.  Among previous studies, we would like to  mention   an illustrative example  in \cite{IGT}  of the Mackey-Glass type diffusive equation with a single delay and with a piece-wise linear birth function. 

Theorem \ref{T1} will be proved in the next section within the framework of the singular perturbation theory developed by Faria {\it et al} in \cite{fhw,FT,FTnl}.    
First, in Subsection \ref{S21}, we analyze the unique heteroclinic solution $u_*(t)$ of equation (\ref{NB}).  We show that the asymptotic Dirichlet series 
approximating $u_*(t)$ at $-\infty$ is uniformly convergent on a sufficiently long time interval. Then we use these approximations to detect parameters $(\tau, p)$ 
for which $u_*(t)$ is a non-monotone (but eventually monotone at $\pm \infty$) and non-oscillating solution of (\ref{NB}). Next, in Subsection \ref{S22}, we extend the aforementioned 
properties of $u_*(t)$ on the wavefront profiles $\phi_c(t)$  for all sufficiently large speeds $c$. The key technical result of this subsection is Lemma \ref{so} which excludes  the existence 
of profiles $\phi_c(t)$  slowly oscillating around the steady state $u=\ln p$ and super-exponentially converging to $u=\ln p$ at $+\infty$.  Lemma \ref{so} shows that  the Mallet-Paret--Cao--Arino theory of {\sl super-exponential solutions}  \cite{OA,Cao,mp} for scalar equations can be extended for some classes of second order delay differential  equations.  Lemma \ref{so}  also helps to show  that, in difference with the monotone wavefronts, eventually monotone wavefronts are linearly determined:  
\begin{thm} \label{T34}  Let $u=\phi(x+ct)$ be a wavefront for equation (\ref{NBD}).  Then the profile $\phi(t)$ is eventually monotone at $+\infty$ if and only if the 
characteristic function  $\chi_+(z,c)= z^2-cz-1 - Pe^{-zc\tau}, \ P:=\ln p -1$,  at the equilibrium  $u = \ln p$ has at least one negative zero. 
\end{thm}
Note that, by \cite[Theorem 1]{TTT}, the leading edge of the profile $\phi(t)$ is  strictly increasing till its first intersection with the equilibrium  level $u=\ln p$. 
Eventual monotonicity criterion of Theorem \ref{T34} complements previous information concerning   the  shapes  of waves for the Nicholson's diffusive equation, cf. \cite{GTLMS,LLLM,TTT}.    It is worth mentioning that  our proofs use in an essential way  several specific characteristics  of the blowflies birth function (in particular, its unimodal form and the negativity of its Schwarz derivative).  There are various other monostable population models with unimodal birth functions (e.g., see \cite[Table 1]{BY},  the cases with overcompensating density dependence). We believe that the problem of the nm-waves in these models can be approached by using techniques from the present  work. For certain,  this does not mean 
that each population model with overcompensating density dependence \cite{BY} necessarily possesses a nm-wave.  

Eventual monotonicity criterion of Theorem \ref{T34} is proved in Section \ref{S223} as Corollary \ref{FC}.  We did not succeed to demonstrate this result by using well-known methods  of upper and lower solutions or  global continuation of waves (these methods were  quite efficient in establishing criteria of wave's monotonicity on $\R$  \cite{FZ,GTLMS,TPTy}). Instead,  we have combined  the above mentioned  Lemma \ref{so} with wavefront's existence and oscillation results from \cite{TTT,TT}  (stated
as Proposition \ref{main} in Appendix). With such an approach, the main technical difficulty was establishing a connection between two different series 
of conditions  (the first one given in Theorem \ref{T34} and the second one given in Proposition \ref{main}).  The required relation (in the form of somewhat cumbersome  inequality (\ref{ck+}))
is proved in Appendix as Lemma \ref{CK}. 
\section{Existence of non-monotone and non-oscillating wavefronts}
\subsection{On the approximation of the heteroclinic connections for the  blowflies equation}\label{S21}
In this subsection, we  establish several key properties of the positive heteroclinic connections to  the Nicholson's blowflies delayed equation 
\begin{equation}\label{defn}
u'(t) = -u(t)+ f(u(t-\tau)), \quad f(u) = pue^{-u}, \quad p >1. 
\end{equation}
The existence and uniqueness of these connections was previously demonstrated  in \cite{FT,AMC}, yet the mentioned works did not provide acceptable analytical tools to approximate the unique heteroclinic  solution $u^*(t)$ (normalized at $-\infty$)  on a given time interval. Sufficiently sharp  approximations are however necessary to prove the existence of the nm-waves, cf. \cite{HKarxiv}. To obtain such an approximation, we are  using  here the asymptotic Dirichlet series representing $u^*(t)$ at $-\infty$. It can be deduced from the  Murovtsev theory \cite{AMR} that this series is  uniformly converging on some infinite interval $(-\infty, s)$.  In the next theorem we are trying to find  $s$ as large as possible by  realizing a direct estimation of the Dirichlet series coefficients (in \cite{AMR},  the method of majorization of $u^*(t)$ by analytic functions was used). 
\begin{thm} \label{T3} Suppose that $p >1, \ \tau > 0,$ and   $\epsilon \in (0, e^{\mu\tau}-1)$.  
Then equation (\ref{defn}) has a  unique (up to translation) positive solution $u^*(t)$ defined for all $t \in \R$  satisfying  $u^*(-\infty)=0$. In addition, 
$u^*(t)$ is bounded, $\liminf_{t \to +\infty} u^*(t) >0$, and 
$u^*(t)$ is a real analytic function.  Moreover, the  solution $u^*(t)$  at $-\infty$ can be represented, modulo an appropriate shift in time, by the Dirichlet series
\begin{equation}\label{Zns}
u^*(t) =  e^{\mu t} +\bar q_2 e^{2\mu t} + \bar q_3e^{3\mu t} + \dots + \bar q_n e^{n\mu t} + \dots,   \quad t \to -\infty, 
\end{equation}
absolutely and uniformly converging on closed subsets of  the interval 
$$(-\infty, T):= \left(-\infty,\tau+\mu^{-1}\ln\left[\frac{\epsilon}{1+\epsilon}\ln\left(1+\frac{1}{|\bar q_2|(1+\epsilon)}\right)\right]\right).$$  
The coefficients $\bar q_j$, $j \geq 2$, alternate in sign and can be calculated recursively from equation (\ref{defn}). In particular, 
$
\bar q_2 = - pe^{-2\mu\tau}/\chi(2\mu)<0,\,
 \bar q_3= p(0.5-2\bar q_2)e^{-3\mu\tau}/\chi(3\mu)>0$, with $\chi$ defined in \eqref{zeta}.
\end{thm}
\begin{proof}  We look for an analytic solution $u=u(t)$ of equation (\ref{defn}) in the form 
\begin{equation}\label{zns}
u(t) =  q_1e^{\mu t} +q_2 e^{2\mu t} + q_3e^{3\mu t} + \dots + q_n e^{n\mu t} + \dots,   
\end{equation}
with some $q_1>0$.  After comparing the coefficients of $e^{(n+1)\mu t}$ in  both sides of the next equation (obtained from \eqref{defn}  by using  a series representation for $f(u)$)
$$
u'(t)+u(t) -pu(t-\tau) =- p\left(u^2(t-\tau) - \frac{u^3(t-\tau)}{2!} + \dots +   (-1)^{n+1}\frac{u^{n+1}(t-\tau)}{n!}+ \dots\right)
$$
we get the recurrent formula   determining $q_{n+1}$, $n \geq 1$:
$$
q_{n+1}= -\frac{pe^{-(n+1)\mu\tau}}{\chi((n+1)\mu)}\left[\sum_{i_1+i_2=n+1} q_{i_1}q_{i_2} -\frac{1}{2!} \sum_{i_1+i_2+i_3=n+1} q_{i_1}q_{i_2}q_{i_3}+\dots+\frac{(-1)^{n+1}}{n!}q_1^{n+1} \right]. 
$$
The general term  of the sum in brackets  is 
$$
A_k:= \frac{(-1)^{k}}{(k-1)!}\sum_{i_1+i_2\dots+i_k=n+1} q_{i_1}q_{i_2}\dots  q_{i_k}, \quad k \geq 2. 
$$
Assuming that $(-1)^{j+1}q_j >0$ for all $j \leq n$, we find that $\mbox{sign}\,A_k$ is
$$
 \mbox{sign}\,\frac{(-1)^{n+1+k+k}}{(k-1)!}\sum_{i_1+i_2\dots+i_k=n+1} (-1)^{i_1+1}q_{i_1}(-1)^{i_2+1}q_{i_2}\dots (-1)^{i_k+1} q_{i_k}= \mbox{sign}\,(-1)^{n+1}. 
$$
Therefore $\mbox{sign}\,q_{n+1} = - \mbox{sign}\,A_k= \mbox{sign}\,(-1)^{n}$.  Notice here that $\chi((n+1)\mu) \geq \chi(2\mu) >0$ for all $n \in \N$. 

Suppose now that $|q_j| \leq q_1:=\sigma$ for each $j =1, \dots,n$.  Then, invoking elementary combinatorics together with the Egorychev method \cite{Ego},  we find that,  for each $\epsilon >0$,
$$
|q_{n+1}|\leq  \frac{pe^{-(n+1)\mu\tau}}{\chi((n+1)\mu)}\left[\sum_{i_1+i_2=n+1} \sigma^2+ \frac{1}{2!} \sum_{i_1+i_2+i_3=n+1} \sigma^3+\dots+\frac{\sigma^{n+1}}{n!} \right]=
$$
$$
 \frac{pe^{-(n+1)\mu\tau}}{\chi((n+1)\mu)}\left[\sigma^2\binom{n}{1}+ \frac{\sigma^3}{2!} \binom{n}{2}+\dots+\frac{\sigma^{n+1}}{n!}\binom{n}{n} \right]= 
$$
$$
 \frac{pe^{-(n+1)\mu\tau}}{\chi((n+1)\mu)}\frac{1}{2\pi i}\oint_{|z|=\epsilon}(1+z)^n\left[\frac{\sigma^2}{z^2}+ \frac{\sigma^3}{z^32!} +\dots+\frac{\sigma^{n+1}}{z^{n+1}n!} \right]dz\leq
$$
$$
 \frac{pe^{-(n+1)\mu\tau}}{\chi((n+1)\mu)}\frac{1}{2\pi i}\oint_{|z|=\epsilon}\frac{\sigma(1+z)^n}{z}\left[\frac{\sigma}{z}+ \frac{\sigma^2}{z^22!} +\dots+\frac{\sigma^n}{z^{n}n!}+\dots \right]dz\leq 
$$
$$
 \frac{\sigma pe^{-(n+1)\mu\tau}}{\chi((n+1)\mu)}(e^{\sigma/\epsilon}-1)(1+\epsilon)^n \leq \sigma |\bar q_2|(e^{\sigma/\epsilon}-1)e^{\mu\tau} (e^{-\mu\tau}(1+\epsilon))^n \leq \sigma
$$
whenever $|\bar q_2|(e^{\sigma/\epsilon}-1)(1+\epsilon)\leq 1$ and $e^{-\mu\tau}(1+\epsilon) <1$. Thus,  for each $\epsilon \in (0, e^{\mu\tau}-1)$,  the Dirichlet 
series (\ref{zns}) with $q_1 =\epsilon\ln(1+[|\bar q_2|(1+\epsilon)]^{-1})$   converges absolutely and uniformly on each closed subset of 
$(-\infty, \tau -\mu^{-1}\ln(1+\epsilon))$. Equivalently, the shifted series 
$$
u^*(t):= u(t-\mu^{-1}\ln q_1) =  e^{\mu t} +\frac{q_2}{q^2_1} e^{2\mu t} + \frac{q_3}{q_1^3}e^{3\mu t} + \dots + \frac{q_n}{q_1^n} e^{n\mu t} + \dots  $$
converges for all $t < T= \tau -\mu^{-1}\ln(1+\epsilon)+\mu^{-1}\ln q_1$. Hence,  equation  (\ref{defn}) has an analytic solution $u^*(t)$ defined and positive on some interval $(-\infty,T_1] \subset 
(-\infty,T)$.  Clearly, integrating  (\ref{defn}) step by step, we can extend this solution for all $t \in \R$. Next, the set $X:= C([-\tau,0], (0,+\infty))$ is invariant with respect to the semi-flow $\Phi^t$ generated by  (\ref{defn}) so  that $u^*(t) >0$ for all $t \in \R$.  Furthermore,  since all solutions to (\ref{defn}) are uniformly persistent when $p>1$ \cite[Theorem 2.4]{BBI}, the trajectory $\Phi^tu^*_0$, where $u^*_0(s) = u^*(s),\ s \in [-\tau,0]$,  has a compact $\omega$-limit set in $X$.  In particular,  $u^*(t)$ is bounded and $\liminf_{t \to +\infty} u^*(t) >0$.  Then by the classical Nussbaum theorem \cite{Nus}, $u^*$ is a real analytic function on $\R$. Finally, the uniqueness of $u^*$ was established in the proofs of  \cite[Lemmas 6 and 8]{FT}.  Alternatively, it can be deduced from the general uniqueness Theorem 3 in \cite{AGT}. \qed
\end{proof}
\begin{remark} Suppose that either $1< p \leq e^2$ or $p> e^2$ and 
\begin{equation}\label{gsc}
e^{-\tau} > P\ln\frac{P^2+P}{P^2+1}, \quad \mbox{where} \ P =\ln p -1. 
\end{equation}
Then \cite[Theorem 2.1]{LTT} implies that the positive solution $u^*(t)$ given in Theorem \ref{T3} converges at $+\infty$: $u^*(+\infty)=\ln p$.
\end{remark}
\begin{example} As in Example \ref{Ex1}, take $\tau =0.07$ and  $p =365$ and choose $\epsilon = 2.2$ $< e^{\mu \tau}-1 = 9.536\dots $.  Then Theorem \ref{T3} provides  the convergence interval $(-\infty,   
0.079] \supset \R_-$ for the series (\ref{Zns}).  In addition,  in  this particular case,  $|\bar q_j|$ are decreasing which suggests the following estimates for 
$u^*(t)$ in the spirit of the alternating series test: 
\begin{equation}\label{ast}
u_2(t):=e^{\mu t} +\bar q_2 e^{2\mu t} < u^*(t) < e^{\mu t}=:u_1(t), \quad t \leq 0. 
\end{equation}
Our next result shows that this estimate is  indeed true. 
 \end{example}
\begin{thm} Solution $u^*(t)$ described in Theorem \ref{T3} and normalized at $-\infty$ by (\ref{Zns}) satisfies  (\ref{ast}).   
\end{thm}
\begin{proof} 
In view of  Theorem \ref{T3},  the inequalities (\ref{ast}) hold on some interval $(-\infty, L)$ with $L\leq 0$. It is easy to see that $u^*(t) < u_1(t)$ 
for all $t \in \R$. Indeed, suppose that  $u_1(a)=u^*(a)$ at some leftmost point $a$. Then $u'_1(a)\leq (u^*)'(a)$, $u_1(a-\tau)> u^*(a-\tau)$,  so that 
$$
0 = (u^*)'(a) +u^*(a)- f(u^*(a-\tau)) >  (u^*)'(a) +u^*(a)- pu^*(a-\tau) > $$
$$ (u_1)'(a) +u_1(a)- pu_1(a-\tau)=0, 
$$
a contradiction. Hence $u^*(t) < e^{\mu t}$ for all  $t \in \R$. In particular, $u^*(t) <1$ for all $ t \leq 0$. 

Next, we have that 
$$
u_2'(t) +u_2(t)- f(u_2(t-\tau)) < u'_2(t)+u_2(t)-p(u_2(t-\tau)-u_2^2(t-\tau))= 
p\bar q_2e^{3\mu (t-\tau)}(2+\bar q_2 e^{\mu (t-\tau)}). 
$$
Since $u_2(t) >0$ if and only if  $1+\bar q_2 e^{\mu t} >0$, we conclude that $u_2'(t) +u_2(t) -f(u_2(t-\tau))<0$ 
whenever $u_2(t) >0$. 

We claim that 
actually the first inequality in (\ref{ast}) also  holds  for all $t \leq 0$. Since $u^*(t) >0$ for all $t\in \R$, it suffices  to prove it only for  those  $t<0$ for which  $u_2(t)>0$. 
Arguing by contradiction again, suppose that $u_2(a)=u^*(a)>0$ at some leftmost point $a\leq 0$. Then $u'_2(a)\geq (u^*)'(a)$, $u_2(a-\tau)< u^*(a-\tau) <1$  so that 
$$
0 = (u^*)'(a) +u^*(a)- f(u^*(a-\tau))  <  (u_2)'(a) +u_2(a)-f(u_2(a-\tau)) <0, 
$$
a contradiction. Observe here that the strict monotonicity of $f(x)=pxe^{-x}$ on the interval $[0,1]$ plays an essential role in the last argument. 
\qed
\end{proof}
We will also  need the following property of $u^*(t)$: 
\begin{lem}\label{Le7}
$(u^*)'(t)>0$  on some maximal interval $(-\infty,M)\subseteq \R$.  Clearly, if $M =+\infty$, then $u^*(M) =\ln p$. If $M$ is finite, then $u^*(M) >\ln p$.   
\end{lem}
\begin{proof}
By differentiating (\ref{Zns}), we find that $(u^*)'(t) = \mu e^{\mu t}(1+ o(1)) > 0$ at $t =-\infty$, from which the existence of the above mentioned $M$ follows. Suppose that 
$M$ is finite, then $(u^*)'(M) =0, $ $(u^*)''(M) \leq 0, $ $(u^*)'(M-\tau) > 0$. Since  (\ref{defn}) implies the relation 
$(u^*)''(M) = f'(u^*(M-\tau))(u^*)'(M-\tau)$, we obtain $f'(u^*(M-\tau))\leq 0$ so that  $u^*(M-\tau)\geq 1$. Clearly, if also $u^*(M-\tau)\geq \ln p$ then 
$u^*(M)> u^*(M-\tau)\geq \ln p$ and the lemma is proved in such a case. On the other hand, if $u^*(M-\tau) \in [1,\ln p)$  then the statement follows from the relation 
$u^*(M) = f(u^*(M-\tau)) > \ln p$.  
\qed
\end{proof}
Now, 
looking for wavefronts to (\ref{NBD}) in the form 
$$
u(t,x) = \phi(\sqrt{\epsilon} x + t), \quad \epsilon = c^{-2}, 
$$
we obtain the profile equation 
\begin{equation}\label{def}
\epsilon \phi''(t) -\phi'(t) - \phi(t) +  f(\phi(t-\tau))=0. 
\end{equation}
Equation (\ref{def}) was analyzed  in \cite{AGT,FT,AMC,GTLMS,LM,LLLM,SZ,TTT}.  Since the first derivative of $f(x)=pxe^{-x}$ is dominated by $f'(0)=p$, 
$|f'(x)| \leq p$ for all $x \geq 0$,  the uniqueness (up to translation) of each semi-wavefront $\phi(t,\epsilon)$ to (\ref{def})  follows from \cite[Theorem 7]{AGT}.   
Next, by \cite{SZ} or \cite[Corollary 12]{TTT} if $1 < p \leq e$ (i.e. if $f(x)$ is monotone on $[0,\ln p]$) then $\phi'(t,\epsilon)>0$ for all $t \in \R$. Now, when $p \in (e,e^2]$, 
$f(x)$ is no longer monotone on  $[0,\ln p]$. Nevertheless, in such a case  $f'(\ln p) \leq f'(x) \leq f'(0)$ for all $x \in [0,\ln p]$ and $f(x)$ {satisfies the feedback }condition on the interval {$(f(p/e),p/e)\setminus \ln p$} (see condition (\ref{fco}) in Appendix). Therefore \cite{GTLMS,TTT} assure that   if $e < p \leq e^2$ then $\phi(t,\epsilon)$ is either monotone or sine-like slowly oscillating around $\ln p$ at $+\infty$;   moreover, the type of monotonicity/oscillation property  is linearly determined.   Similarly, \cite[Theorem 2.5]{GTLMS} implies that $\phi(t,\epsilon)$ is strictly increasing when $p >e^2$ and $p\tau e^{\tau -1} \leq 1$. Next, it is easy to see that the characteristic equation $\chi_+(z,\epsilon): = \epsilon z^2-z -1 -P e^{-z\tau}$ does not have real negative roots when $P\tau e^{1+\tau} \geq 1$. Then 
\cite[Lemma 25]{TTT} implies that $\phi(t,\epsilon)$ cannot be eventually monotone if $P\tau e^{1+\tau} \geq 1$. Hence, we have established the following result: 
\begin{lem} \label{Le77} Suppose that equation (\ref{NBD}) has a nm-wave. Then  necessarily $p > e^2$, $P\tau e^{1+\tau} < 1$ and $p\tau e^{\tau -1} > 1$. In particular,
$\tau \in (0,\tau_*)$ where $\tau_*= 0.278\dots$ solves equation $\tau e^{1+\tau} =1$. 
\end{lem}
\begin{cor} \label{Co8} If equation (\ref{NBD}) has a nm-wave then $e^{-\mu\tau} <0.5$ and $\bar q_2 \in (-1,0)$. 
\end{cor}
\begin{proof}
Indeed, $\mu+1= pe^{-\mu \tau} > e^{1-(\mu+1) \tau}/\tau$. Set $\rho = \tau(\mu+1)$, then $\rho > e^{-\rho+1}$ implying that $\rho >1$.  Consequently, $-\tau \mu < \tau -1$ and $e^{-\tau \mu} < e^{-1+\tau} <  e^{-1+\tau_*}=0.48\dots$, 
$
-\bar q_2 = pe^{-2\mu\tau}/\chi(2\mu)= pe^{-2\mu\tau}/(\mu+ p(e^{-\mu\tau} -e^{-2\mu\tau}))<  e^{-\mu\tau}/(1 -e^{-\mu\tau}) =0.94 \dots
$ \qed
\end{proof}
In the next stage of our proof, we continue to analyze the heteroclinic solution $u^*(t)$ given in Theorem \ref{T3}.  This time, we will take parameters $(\tau, p)$ as in Lemma \ref{Le77} which will have  important consequences for the form and estimates of $u^*(t)$.  For example, in such a case, due to Corollary \ref{Co8} the lower bound  $u_2(t)$ in (\ref{ast}) is positive for all $t \leq 0$.  Especially we will be interested 
in the oscillation properties (around the equilibrium $\ln p$) of $u^*(t)$. Since $u^*(t)$ is a real analytic function on $\R$, the set of all solutions of equation $u^*(t)=\ln p$
is either an empty set or can be represented as a strictly increasing sequence $\mathcal S=  \{t_j\}$ of positive numbers. By the same reason,  if the sequence $\mathcal S$  is infinite, it should converge to $+\infty$. Our next goal is to prove that $t_j$ are simple zeros of $u^*(t)-\ln p$ and that $t_{j+1} -t_j > \tau$.  The main complication here consists in the fact that the birth function  $f(x)=pxe^{-x}$ for $p > 16.99\dots$ does not satisfy the  feedback condition (\ref{fco}), cf. \cite{GTLMS,TTT}. Our arguments below are inspired by the classical theory of J. Mallet-Paret \cite{mp}.  
\begin{lem} \label{Le9} Suppose that $p > e^2$ and $P\tau e^{1+\tau} < 1$. Then inequality (\ref{gsc}) is satisfied so that $u^*(+\infty)=\ln p$ and, in fact, 
 $u^*(t)$ is eventually monotone solution.   
Furthermore, the above defined sequence  $\{t_j\}$ has at most a  finite number  of elements, $t_1  <t_2 < \dots < t_m$,  $m\in \N$ and  $t_{j+1}-t_j > \tau$, $(u^*)'(t_{j})\not=0$ for each such $t_j$. 
\end{lem}
\begin{proof} It is easy to check that the smooth curves $\gamma_1, \gamma_2\subset \R_+$ defined by equations  $P\tau e^{1+\tau} = 1$ and $e^{-\tau} = P\ln\frac{P^2+P}{P^2+1}$, respectively,  intersect at some point $(\tau_0, P_0)$ if and only if $\tau_0$ satisfies 
equation $e^{\tau e}(1 + \tau^2 e^{2(1+\tau)}) = \tau e^{1+\tau}+1$. However, a straightforward comparison of the Taylor coefficients in this equation shows that $e^{\tau e}  >\tau e^{1+\tau}+1$
for all positive $\tau$. Hence, $\gamma_1\cap\gamma_2=\emptyset$.  Actually, an easy inspection shows that the domain bounded by $\gamma_1$ and the coordinate axes 
lies inside the domain bounded by $\gamma_2$ and the coordinate axes. 

Now, suppose that $t_1$ is the moment of the first intersection of the graph of $u^*(t)$ with the line $u =\ln p$. Then $(u^*)'(t_1) >0$ in view of Lemma \ref{Le7}. Suppose next that 
$t_2$ is  the moment of the second intersection of $u^*(t)$ with the line $u =\ln p$ and $u^*(s_1) : = \max\{u^*(t), \ t \in [t_1,t_2]\}$. Then $(u^*)'(s_1)=0, \ (u^*)'( t_2)\leq 0$, from which we obtain   $\ln p < u^*(s_1) = f(u^*(s_1-\tau))$, $\ln p=u^*(t_2) \geq f(u^*(t_2-\tau))$. In the case when $t_2-t_1\leq \tau$ we have, in addition,  that 
$\ln p \geq u^*(t_2-\tau) > u^*(s_1-\tau)$. Consider first the situation when $\ln p > u^*( t_2-\tau)$. Then  $\ln p \geq f(u^*(t_2-\tau))$ implies that $u^*(t_2-\tau) \leq \min f^{-1}(\ln p)$. 
On the other hand  $\ln p < u^*(s_1) = f(u^*(s_1-\tau))$ implies that  $u^*(s_1-\tau) > \min f^{-1}(\ln p)$. Thus $u^*(t_2-\tau) < u^*(s_1-\tau)$, a contradiction. Finally, suppose 
that $\ln p = u^*(t_2-\tau)=u^*(t_2)$. Then  $(u^*)'(t_2)=0$ implying the following contradictory relation $(u^*)''(t_2)=f'(u^*(t_2-\tau))(u^*)'(t_2-\tau)<0$. The above discussion shows 
that $ t_2-t_1>\tau$ and that $(u^*)'(t_2)<0$. 

By our convention, $u^*(t)< \ln p$ for $t\in (t_2,t_3)$ and $u^*(s_2) = \min\{u^*(t), \ t \in [t_2,t_3]\}$ for some leftmost $s_2\in (t_2,t_3)$. Again, $\ln p > u^*(s_2) = f(u^*(s_2-\tau))$. Since  
$u^*(s_2) < u^*(s_2-\tau)$, the latter implies  that $u^*(s_2-\tau) > \ln p$ and, consequently, that $s_2-t_2 \leq \tau$. 
Therefore, using the variation of constant formula, we get 
$$
u^*(s_2) = e^{-(s_2-t_2)}\ln p +\int_{t_2}^{s_2}e^{-(s_2-s)}f(u^*(s-\tau))ds > e^{-\tau}\ln p \geq e^{-\tau_*}\ln e^2 =  1.51\hbox{$\dots$} >1. 
$$
In particular,  $u^*(t) >1$ for all $t\in (t_2,t_3)$.  Now, since $(u^*)'(t_3) \geq 0$, we find that $\ln p=u^*(t_3) \leq f(u^*(t_3-\tau))$. This implies that $u^*(t_3-\tau) \leq \ln p$. Suppose 
that $u^*(t_3-\tau) = \ln p$, then obviously $t_3-\tau = t_2$ (since $t_1 < t_1+\tau < t_2 < t_3$ are the consecutive zeros of $u^*(t)-\ln p$) and $(u^*)''(t_3) = f'(\ln p)(u^*)'(t_2)>0$, a contradiction.  Consequently,  $u^*(t_3-\tau) < \ln p$, $(u^*)'(t_3) >0$ and $t_3-t_2 > \tau$. 

Next, $u^*(t)> \ln p$ for $t\in (t_3,t_4)$ and $(u^*)'(t_4) \leq 0$. Thus $\ln p=u^*(t_4) \geq f(u^*(t_4-\tau))$ from where  $u^*(t_4-\tau) \geq \ln p$ (here we are using the inequality 
$u^*(t) >1$ for $t \in  [t_1,t_4]$).  Suppose 
that $u^*(t_4-\tau) = \ln p$, then again we get  $t_4-\tau = t_3$ and $(u^*)''(t_4) = f'(\ln p)(u^*)'(t_3)<0$, a contradiction.  Consequently,  $u^*(t_4-\tau) > \ln p$, $(u^*)'(t_4) <0$ and $t_4-t_3 > \tau$. Clearly, the above described inductive steps can be continued to include all  points $t_j$. 

Finally, we will prove that the sequence  $\{t_j\}$ is finite. Linearizing (\ref{NB})  at the equilibrium $\ln p$, 
we find the associated characteristic function 
$
\chi_+(z) := z+1+Pe^{-z\tau}.  
$
After a straightforward computation, we find that  $\chi_+(z)$ has exactly two different real roots $z_2 < z_1<0$ if and only if $P\tau e^{1+\tau} < 1$.  Moreover, in such a case,  each other root 
$z_j =\alpha_j+i\beta_j,\ \beta_j \geq0, \ j >2$,  satisfies the inequalities $\alpha_j < z_2$ and $\beta_j > 2\pi/\tau$, see \cite[Theorem 6.1]{mp}. In particular, the steady state $u = \ln p$ is exponentially stable. 
We claim that $u^*(t)$  is eventually 
monotone at $+\infty$ if $P\tau e^{1+\tau} < 1$. Indeed, otherwise, it is easy to see that the sequence $\{t_j\}$ is infinite so that $u^*(t)$ exponentially converges  to the equilibrium  $\ln p$ and  slowly oscillates around it. Then in view of Yulin Cao results   \cite[Theorem 3.4]{Cao}  on super-exponential solutions,  there is   a zero $z_j=\alpha_j+i\beta_j, j \in \N,$ of $\chi_+(z)$ and $C\not=0, \ \delta >0, \ \theta \in \R,$ such that 
$$
u^*(t)-\ln p = Ce^{\alpha_jt}\cos(\beta_j t+\theta) + O(e^{(\alpha_j-\delta) t}), \quad t \to +\infty. 
$$
Now, if $z_j$ is not a real root,  the above representation implies that the distance between large adjacent zeros of  $u^*(t)-\ln p$ is less than $\tau/2$, a contradiction. Hence\begin{equation}\label{j}
u^*(t)=\ln p + C_*e^{z_jt} + O(e^{(z_j-\delta) t}), \quad t \to +\infty,
\end{equation}
where  $j \in \{1,2\}, z_j <0, \ C_*\not=0$. This completes the proof of Lemma \ref{Le9}. 
 \qed
\end{proof}
Importantly, the sequence $\mathcal S$ can be non-empty for certain parameters $p, \tau$. The next result shows that generally $m \geq 1$ in Lemma \ref{Le9}.  It is an open question, however, whether  there exist $p, \tau$  for which  $m \geq 2$. 
\begin{cor} \label{30} In addition to the assumptions of Lemma \ref{Le9}, suppose that  $\zeta>\ln p$ where $\zeta $ was defined in Theorem \ref{T1}.  
Then the solution $u^*(t) $ is eventually monotone at $+\infty$ and   is non-monotone on $\R$. Moreover, $\max_{\R} u^*(t) > \zeta$. 
\end{cor} 
\begin{proof} Indeed, after using the variation of constant formula  on $[0,\tau]$ and  monotonicity of $f(u)$ on the interval $[0,1]$, we find that 
$$
u^*(\tau) = u^*(0)e^{-\tau} +\int_0^\tau e^{s-\tau}f(u^*(s-\tau))ds > u_2(0)e^{-\tau} +\int_0^\tau e^{s-\tau}f(u_2(s-\tau))ds > 
$$
$$
u_2(0)e^{-\tau} +p\int_0^\tau e^{s-\tau}u_2(s-\tau)e^{-u_1(s-\tau)}ds = 
$$
$$
(1+\bar q_2)e^{-\tau} +p\int_{-\tau}^0 {\color{blue}e^{\mu s} }e^{s}(1+ \bar q_2 e^{\mu s})\exp(-e^{\mu s})ds = \zeta. 
$$
Finally, recall that $\zeta > \ln p$ for the parameters $p=365,$ $\tau =0.07$, see Example \ref{Ex1}. 
\qed
\end{proof}
\subsection{Proof of Theorem \ref{T1}}\label{S22}
The key relation between $u^*(t)$ and  $\phi(t,\epsilon)$  is given in the next assertion. 
\begin{proposition} \label{P1} Let all the conditions of Corollary \ref{30} be satisfied. Then  $\phi(+\infty,\epsilon) = \ln p$ for all sufficiently small $\epsilon >0$. Moreover, $\lim_{\epsilon \to 0^+} \phi(t,\epsilon) = u^*(t)$ uniformly on $\R$
(possibly,  after an appropriate translation of the wavefronts). In particular, the wave profiles $u= \phi(t,\epsilon)$ are not monotone for all small $\epsilon >0$. 
\end{proposition}
\begin{proof} First, suppose that the characteristic equation $\chi(z)=0$ does not have roots on the imaginary axis. 
Then, in view of the uniqueness of semi-wavefronts in the Nicholson's blowflies diffusive equation, the statement concerning the uniform convergence  is a direct consequence  of \cite[Theorem 1]{FT}. 
Now, if the zero equilibrium of equation (\ref{NB}) is not hyperbolic, the same conclusion follows from  Theorem 3.8 in \cite{FTnl}.  
\qed 
\end{proof}
To complete the proof of Theorem \ref{T1},  we have to extend the eventual monotonicity property of $\phi(t,0) = u^*(t)$ established in Lemma \ref{Le9} on the profiles $\phi(t,\epsilon)$ for all sufficiently small $\epsilon > 0$.  Observe that, taking into account Proposition \ref{P1} and using  \cite[Theorem 2.1]{mps} we obtain (see \cite[pp. 2321-2322]{TTT} for more detail) that $\phi(t,\epsilon)$ is either eventually monotone  or  slowly oscillating around $\ln p$ at $+\infty$ in the following sense: 
\begin{definition} \label{d2s} Set $\mathbb{K} = [-\tau,0] \cup \{1\}$. For any
$v \in C(\mathbb{K})\setminus\{0\}$ we define the number of sign
changes by $$\hspace{-1mm} {\rm sc}(v) = \sup\{k \geq 1:{\rm  there
\ are \ } t_0 <
 \dots < t_k, \ t_j \in \mathbb{K},   \ {\rm such \ that\ }
v(t_{i-1})v(t_{i}) <0 {\rm \ for \ }  i\geq 1\}. $$ We set ${\rm
sc}(v) =0$ if $v(s) \geq 0$ or  $v(s) \leq 0$ for $s \in
\mathbb{K}$.  Next, for a smooth function $\psi: [T_0, +\infty) \to \R$ and a real number $\kappa$, we will write
$(\overline{\psi}_t)(s) =
\psi(t+s)- \kappa$ if $s \in [-\tau,0]$, and $(\overline{\psi}_t)(1) =
\psi'(t)$. We will say that $\psi(t)$  is slowly oscillating around $\kappa$ on a connected interval $\mathfrak{J}\subset  [T_0+\tau, +\infty) $
 if the following conditions are satisfied: 
 
 \vspace{2mm}
 
 (d1) $\psi(t)$ oscillates around $\kappa$ and

 \vspace{2mm}

 (d2) for each $t \in \mathfrak{J}$, it holds  that either sc$(\overline{
\psi}_t)=1$ or sc$(\overline{\psi}_t)=2$.
\end{definition}
Therefore our immediate goal is to demonstrate that $\phi(t,\epsilon)$ is not a slowly oscillating solution of equation (\ref{def}) whenever all the conditions of Corollary \ref{30} are satisfied.  
We have already solved a similar problem, proving eventual monotonicity of $u^*(t)$ in Lemma \ref{Le9}. 
However,  trying to argue as at the end of the proof of Lemma \ref{Le9}, we will regret  the lack of an analog of the Mallet-Paret--Cao--Arino theory \cite{OA,Cao,mp} of {\sl super-exponential solutions} (i.e. solutions converging to their finite limits at $\infty$ faster than any exponential) for the second order delay differential  equations. The key idea of this theory was concisely described by 
Ovide Arino:  "{\it Surprisingly, the proof $[\dots]$ is based on quite an elementary although very neat observation, which is, essentially,  if you assume a solution has a rapid decay from $t$ to $t+\tau$, it indicates rapid oscillations before $t$"}, see  \cite[p. 178]{OA}. 
 In our next lemma, we  contribute to the analysis of super-exponential solutions by developing this idea for   second order delay differential  equations (at a level of generality sufficient for our purposes).
\begin{lem} \label{so} Let $y(t), \ y(+\infty)=0,$ slowly oscillate around $0$ and satisfy equation 
\begin{equation}\label{f}
y''(t)-Ay'(t)-By(t) - C(t)y(t-\tau) =0, \quad t \in \R,  
\end{equation}
where $A,B >0$ and continuous function $C(t)$ converges to a positive number $C$ at $+\infty$.
Suppose also  that the set of zeros of $y'(t)$ does not contain a non-degenerate interval. Then $y(t)$ is not a super-exponentially decaying (i.e. small) solution.  
\end{lem}
\begin{proof}  Arguing by contradiction, suppose that equation (\ref{f}) has a slowly oscillating small solution $y(t)$. After realizing the change of variables $y(t) = e^{rt}z(t)$, where $r$ is a negative root of equation $r^2-A r -B =0$, 
we obtain the delay differential equation for $z$ of the same type  as  (\ref{f}) (i.e. $a>0$, $c(+\infty)>0$) except that now $B=0$: 
\begin{equation}\label{red}
y''(t)-ay'(t)- c(t)y(t-\tau) =0, \quad t \in \R.
\end{equation}
Obviously, $z(t) = e^{-rt}y(t)$ is super-exponentially  decaying at $+\infty$.  It is easy to check that, being a small solution, $z(t) = e^{-rt}y(t)$ with $r<0$ shares  the slow oscillation property of $y(t)$. We will arrive to a contradiction by supposing  the existence of a slowly oscillating small solution $y(t)$ of  equation (\ref{red}).

Hence, arguing by contradiction, suppose that equation (\ref{red}) has a slowly oscillating small solution $y(t)$.  In our subsequent argumentation, we will use the notation $y_t(s)=y(t+s), \ s \in [-\tau,0]$ so that $y_t\in C[-\tau,0]$ and $\|y_t\|= \max\{|y(t+s)|, s \in [-\tau,0]\}$. It is easy to see that the `smallness' of $y(t)$ implies the  existence of 
an increasing sequence $\{t_j\}$, $t_j \to +\infty$, such that  $\|y_{t_j+2\tau}\|/\|y_{t_j+\tau}\| \to 0$ as $j\to +\infty$ (since $y(t)$ oscillates, we have that $y_t\not= 0$ for each $t$).

Set $z_j(t)= y(t_j+t)/\|y_{t_j+\tau}\|$, clearly $|z_j(t)| \leq 1$ for $t \in [0,\tau]$ and $|z_j(\xi_j)| =1$ for some $\xi_j \in [0,\tau)$. Without loss of generality, we can further assume that 
 $z_j(\xi_j) =1$ and that $\xi_j \to \xi_*$ for some $\xi_*\in [0,\tau]$.  Due to our choice of $\{t_j\}$, we also find  that $z_j(t) \to 0$ uniformly on the interval $[\tau,2\tau]$. In addition, $z_j(t)$ is a slowly oscillating small solution of  the differential equation 
\begin{equation}\label{redj}
z''(t)-az'(t)- c_j(t)z(t-\tau) =0, \ \  t \in \R, \ \ \mbox{where} \ c_j(t) := c(t+t_j) \to c_*: =Ce^{-r\tau}>0. 
\end{equation}
We will say that an interval $(p,q)$ is a  maximal complete interval of monotonicity for some $z_j$ if $z_j(t)$ is monotone on $(p,q)$, $z_j'(p)=z_j'(q)=0$ and $(p,q)$ is not properly 
contained in a larger interval with the same properties. Then there exists some $r_0 \in (p,q)$ such that $0\not =|z_j'(r_0))| =\max\{|z'_j(s)|, s \in (p,q)\}$ and $z''_j(r_0)=0$. Thus 
$$\mbox{sign}\,z_j(r_0-\tau) = \mbox{sign}\,(-az_j'(r_0)/c_j(r_0)) = -\mbox{sign}\,z_j'(r_0)
$$ so that $z_j(t)$ can have at most 3 maximal complete intervals of monotonicity on each interval of  length $\tau$ (consequently, at most 5 intervals of monotonicity). As a consequence, we can find a subsequence $\{z_{j_i}(t)\}$ of $\{z_j(t)\}$  such that that all $z_{j_i}(t)$ have exactly the same number 
of intervals of monotonicity $(p^{(j_i)}_k,q^{(j_i)}_k)$ inside of each of  the segments $[-\tau,0], [0,\tau], [\tau, 2\tau]$, while  the sequences $p^{(j_i)}_k, q^{(j_i)}_k$ are converging to their respective limits 
$p_k, q_k$. Without loss of generality, we can also assume that every $z_{j_i}(t)$ does not change its sign on each of the intervals $(p^{(j_i)}_k,q^{(j_i)}_k)$. 
Furthermore, due tho the Helly selection theorem \cite[p. 250]{BZ}, we can assume that there exists a piece-wise monotone function $z_*(t), \ t \in [0,2\tau],$ such that $z_*(t)=0, \ t \in [\tau,2\tau]$ and  $z_{j_i}(t) \to z_*(t)$ pointwise on $[0,2\tau]$.   In the sequel, to simplify the notation, we will {denote by $z_{j}(t)$ also any subsequence of $\{z_{j}(t)\}$}.

We claim that $z_*(t)=0$ almost everywhere on $[0,\tau]$. First, observe that the sequence of the total variations $\int_{\tau}^{2\tau}|z_j'(s)|ds$ of the smooth functions $z_j(t)$ on $[\tau, 2\tau]$ converges to 0.  Then, by the Riesz theorem \cite[p. 79]{BZ}, we can extract a subsequence $z'_{j}(t)$ such that  $z_{j}'(t) \to 0$ almost everywhere  on $[\tau,2\tau]$. 
Take some $s_1<s_2, s_k \in [\tau, 2\tau]$, $k=1, 2$, such that 
$\lim_{j \to \infty} z_j'(s_k) = 0$. By integrating (\ref{redj}), we find that 
\begin{equation}\label{vc}
z'_j(s_2) = e^{a(s_2-s_1)}z_j'(s_1) + \int_{s_1}^{s_2}e^{a(s_2-s)}c_j(s)z_j(s-\tau)ds. 
\end{equation}
Passing to the limit  $j \to +\infty$ in (\ref{vc}), 
we find that, for almost all $s_1<s_2,  s_k \in [\tau, 2\tau]$, $k=1,2,$
$$
\int_{s_1}^{s_2}e^{a(s_2-s)}z_*(s-\tau)ds =0,
$$ 
which proves our claim. 

Now, fix some non-empty monotonicity  interval $(p_k,q_k)\subset [0,\tau]$ mentioned before.   Due to the claim of the previous paragraph, we know that there are $r_1, r_2 \in (p_k, q_k)$ such that 
 $\lim_{j \to \infty} z_j(r_k) = 0$. Then we find immediately that $\int_{r_1}^{r_2}|z'_j(t)|dt \to 0$ and $z_j(t) \to 0$ uniformly on $[r_1,r_2]$. 
Arguing now as in the previous paragraph and passing to subsequences if necessary, we  can conclude that $z_j'(t) \to 0$ almost everywhere on $(p_k,q_k)$.  Hence, we have 
proved that there exists a finite set of points $\frak{R}={\{p_k, q_j: \ p_k, q_j \in [0,\tau]\} }$ such that $z_j(t)\to 0$ uniformly on each closed interval ${\mathcal I} \subset
[0,\tau] \setminus \frak{R}$. Consequently,  $z_j'(t) \to 0$ almost everywhere on $[0, \tau]$ and $\xi_* \in \frak{R}$. 

In the next stage, we will analyze the sequence $\{z_j(t)\}$ on the interval $[-\tau,0]$. Set $\frak{R}_1= \{p_k,q_j:\ p_k, q_j \in [-\tau,0]\}$  and consider some  closed interval $[\alpha,\beta] \subset
[-\tau,0] \setminus \frak{R}_1$.  For all sufficiently large $j$,  the functions $z_j(t)$ have the same type of positivity and monotonicity  on $[\alpha,\beta]$. We claim that we can choose a subsequence in such a way that 
$z_{j}(t) \to 0$  uniformly on $[\alpha,\beta]$. To be more specific,  suppose, for example,  that $z_j(t)$
are non-negative and decreasing on $[\alpha, \beta]$.  If some subsequence of $\{z_j(s)\}$ converges to $+\infty$ for  $s \in (\alpha, \beta)$,  then clearly 
$z_{j}(t) \to +\infty$ uniformly on  $[\alpha, s]$ while $z_{j}'(t) \to 0$  almost everywhere on $[\alpha+ \tau, s +\tau]$. But then, after taking  some $s_1<s_2, s_k \in [\alpha+ \tau, s +\tau]$, $k=1,2$, such that 
$\lim_{j \to \infty} z_{j}'(s_k) = 0$ and using (\ref{vc}), 
we immediately get a contradiction. This implies that the sequence $z_{j}(t)$ is bounded for each $t \in (\alpha, \beta)$. In fact, since we can slightly move the endpoints $\alpha, \beta$ of this interval, the sequences  $z_{j}(\alpha), z_j(\beta)$ are also bounded.  This allows us to conclude that $\{z_j(t)\}$ is uniformly bounded on $[\alpha,\beta]$. Consequently, by arguing as above and passing to subsequences if necessary, we find that $z_j(t)\to 0$ uniformly on $[\alpha,\beta]$ and $z_j'(t)\to 0$  almost everywhere on $[\alpha,\beta]$.  
In other words, $z_j(t), z_j'(t)$ have similar convergence properties on the intervals $[0,\tau]$ and $[-\tau,0]$. 

Therefore, in each $\delta$-neighborhood of $\xi_*$, we can find points $t_1 < \xi_* <t_2, \theta_j$ such that $\lim z_j(t_k) = \lim z'_j(t_k)=0$, $k=1,2,$
 $t_1 < \theta_j \leq \xi_j < t_2$  (whenever $j$ is sufficiently large) and  $z'_j(\theta_j) =0$, 
$z''_j(\theta_j) \leq 0$,  $z_j(\theta_j) \geq 1$. Similarly, the  interval 
$[\xi_*-\tau+\delta, \xi_*-\tau+2\delta]$ contains  a point  $t_3> t_2-\tau$ such that 
$\lim z_j(t_3) = \lim z'_j(t_3)=0$.

But then,  if $\delta \in (0,0.5)$,  there are some  points $s_j, S_j$ satisfying 
the inequalities $t_1 < s_j < \theta_j < S_j < t_2$ and the relations $z_j'(s_j)  > 2, z_j''(s_j)  = 0$, $z_j'(S_j)  < -2, z_j''(S_j)  = 0$.  Since $c_j(t_j) \to c_*>0$, we  can also assume 
that $0< c_j(t)  < 2c_*$. 
Consequently, in view of equation (\ref{redj}),
$$
z_j(\theta_j-\tau)  \leq 0, \quad z_j(s_j-\tau) = -a z'(s_j)/c_j(s_j) < -2a/c_j(s_j) < -a/c_* < 0, 
$$
$$
 z_j(S_j-\tau) = -a z'(S_j)/c_j(S_j) >  2a/c_j(S_j) > a/c_* > 0. 
$$
Now, since $\xi_* - \delta - \tau  <  \theta_j -\tau < S_j-\tau < \xi_*+\delta -\tau < t_3< \xi_*+2\delta -\tau$, for  each $j$  big enough we can find points 
$M_j,T_j\in [\theta_j-\tau, t_3]$ such that $M_j< T_j$, $z_j(M_j) = \max\{z_j(s), s \in [\theta_j-\tau, t_3]\}$,   $z_j'(T_j)  < -2a/(3c_*),\  z_j''(T_j)  = 0$. Thus
$
z_j(T_j-\tau) = -az'(T_3)/c_j(T_3) >0. 
$
Finally, consider four points $T_j-\tau < s_j-\tau < M_j < T_j$. Since $z_j(T_j-\tau) >0,$ $z_j(s_j-\tau) <0, $ $ z_j(M_j) >0, z'_j(T_j)<0$, we conclude that sc$((\bar z_j)_T) >2$ so that  $z_j(t)$ does not slowly oscillate  around $0$.  The obtained contradiction completes the proof of Lemma \ref{so}. 
\qed
\end{proof}

\begin{cor} \label{28} Assume all the conditions of Corollary \ref{30}. Then there exists $\epsilon_0>0$ such that the solution $\phi(t,\epsilon)$ is eventually monotone at $+\infty$ for each $\epsilon\in [0,\epsilon_0]. $
\end{cor}
\begin{proof}
After linearizing (\ref{def})  at the equilibrium  $\phi(t) =\ln p$, 
we find the related characteristic equation
\begin{equation}\label{ehe}
\chi_{+}(z,\epsilon)= \epsilon z^2 -z -1 - Pe^{-z\tau}=0. 
\end{equation}
If  $P\tau e^{1+\tau} < 1$, it follows from \cite[Lemma 1.1]{GTLMS} and \cite[Lemma 2.1]{TT} that there are $\delta >0$ and $\epsilon_1 >0$ such that (\ref{ehe}) has in the half-plane 
$\Re z > z_2-2\delta$ for each $\epsilon \in (0,\epsilon_1]$ exactly three roots $z_j(\epsilon), \ j =0,1,2$.  Moreover, these roots are real and $z_j(\epsilon) \to z_j, \ j =1,2$, and $z_0(\epsilon) \to +\infty$ as $\epsilon \to 0^-$ (as in Lemma \ref{Le9}, here $z_j$ denote the zeros of $\chi_+(z,0)$). 
Therefore the steady state $\ln p$ of (\ref{def}) is hyperbolic and the orbit associated with  $\phi(t,\epsilon)$  belongs to the stable manifold of $\phi(t) =\ln p$.  Thus
 $y(t):= \phi(t,\epsilon)- \ln p$ and $y'(t)$ have at least the exponential rate of decay at $+\infty$. Now, 
arguing by contradiction, suppose that the function $y(t) = \phi(t,\epsilon) -\ln p$ oscillates slowly around $0$.  Clearly,  $y(t)$ satisfies equation (\ref{f}) where 
$A =B  = \epsilon^{-1}, $ $C(t) = -\epsilon^{-1}\int_0^1f'(\phi(t,\epsilon)s+ (1-s)\ln p)ds$, $C(+\infty) =-\epsilon^{-1} f'(\ln p) =\epsilon^{-1} P >0$.  But then Lemma 
\ref{so} guarantees that $y(t)$ has at most the exponential rate of decay at $+\infty$. 
This implies that  there is   a complex zero $z_j(\epsilon)=\alpha_j(\epsilon)+i\beta_j(\epsilon), \ \beta_j(\epsilon) > 0, \ \alpha_j(\epsilon) < z_2(\epsilon)$,  of $\chi_{+}(z,\epsilon)$ and $d(\epsilon) \not=0, \ \delta >0, \ \theta \in \R,$ such that 
$$
\phi(t,\epsilon) -\ln p = d(\epsilon) e^{\alpha_j(\epsilon)t}\cos(\beta_j(\epsilon) t+\theta) + O(e^{(\alpha_j(\epsilon)-\delta) t}), \quad t \to +\infty. 
$$
Now, by Lemma 21 in \cite{TTT}, $\beta_j(\epsilon) \geq 2\pi/\tau$ and therefore $y(t)$  does not oscillate slowly around $0$. The obtained contradiction 
 completes the proof of Corollary \ref{28}. 
\qed 
\end{proof}
Clearly, Theorem \ref{T1} is a direct consequence of Proposition \ref{P1} and Corollaries \ref{30} and \ref{28}.   
\begin{remark}\label{RONE} 
\begin{figure}[h] \label{F41}
\centering \fbox{\includegraphics[width=13.5cm]{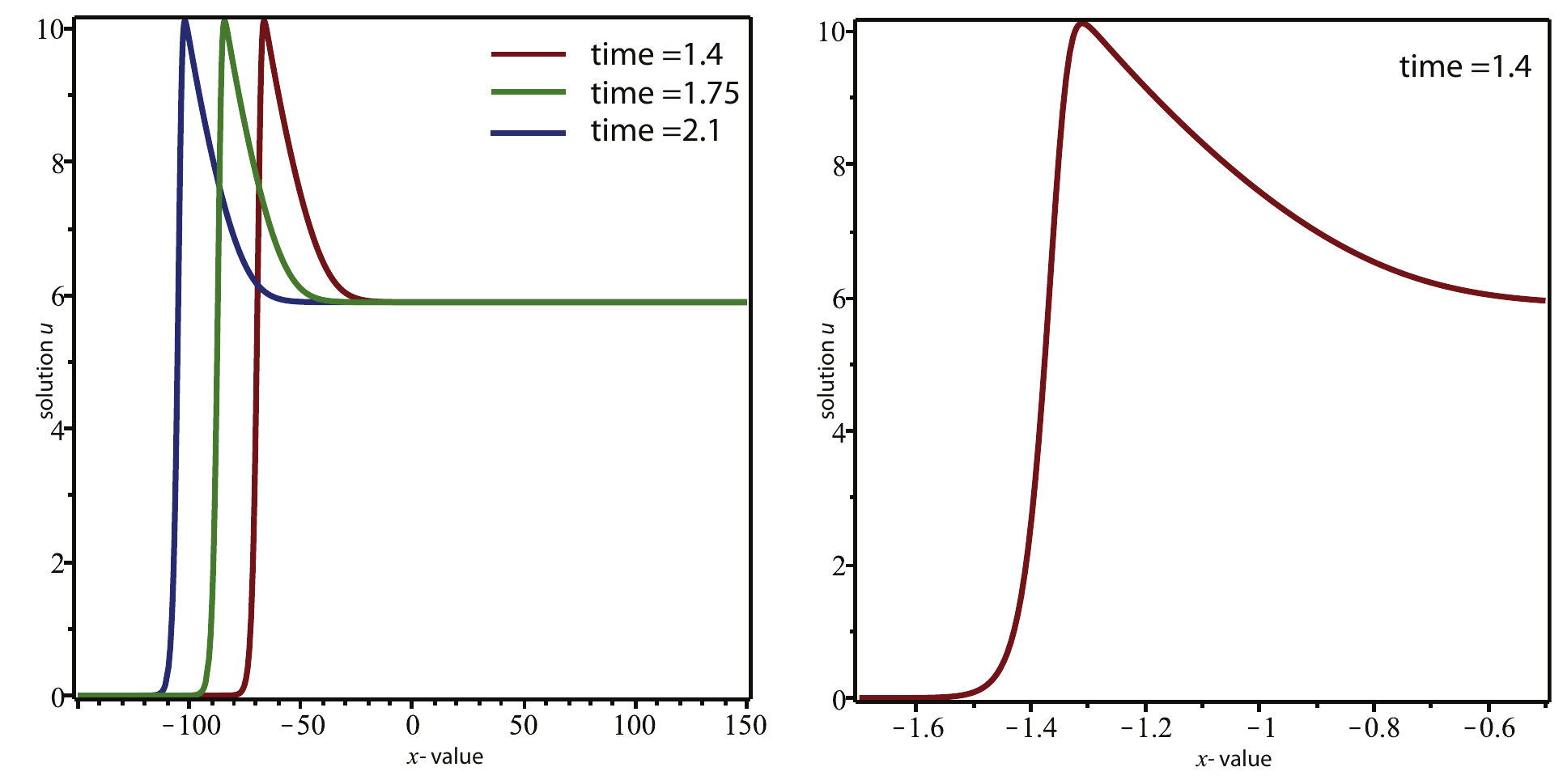}}
\caption{\hspace{0cm}  Equation  (\ref{NBD}) with  $\tau=0.07$, $p=365$:  numerical approximations of the wavefront $u(t,x)=\phi_c(x+ct)$ moving leftward with the speed $c \approx 50$ (left); the rescaled profile $\phi_c(tc)$ for $c\approx 50$ (right). } 
\end{figure} On Fig. 4, we present a fast traveling wave for equation (\ref{NBD})  with  $\tau=0.07$, $p=365$.
The numerical simulations are based on the Crank-Nicholson method which is second-order
accurate in both spatial and temporal directions. The initial function is  
$$
u_0(t,x)= \left\{ \begin{array}{ll}
\exp(0.7x), & \mbox{as} \,\,\, x<0, \enspace t \in [-\tau,0], \\
\ln(p), & \mbox{as} \,\,\,  x \geq 0, \enspace t \in [-\tau,0].
\end{array}
\right.
$$
The spatial step size is chosen as $\Delta x=0.05$ in the computational domain $x \in [-150,150]$ where the Dirichlet boundary conditions $u(t,-150)=0$ and $u(t,150)=\ln p$ are imposed.  The temporal step size is $\Delta t=0.01$. After some short initial period, the numerical solution $u(t,x)$ behaves like a nm-wave moving leftward with the speed $c\approx 50$. This provides  an additional numerical confirmation of the theoretical result in \cite{LLLM} in the case of nm-waves (theoretical speed of propagation 
is $c =48.26\dots$).  Observe that the rescaled profile $\phi_c(ct), \ c \approx 50,$ is rather well approximates the limit profile, $c =\infty$,  on the Fig.1.
\end{remark}
\subsection{Proof of Theorem \ref{T34}}\label{S223}
\vspace{3mm}
\noindent Theorem \ref{T34} follows from the following equivalent statement:
\begin{cor}\label{FC}
 Let $u(t,x)=\phi(x+ct)$ be a wavefront for equation (\ref{NBD}).  Then the profile $\phi(t)$ is eventually monotone at $+\infty$ if and only if the 
characteristic function  $\chi_+(z,\epsilon)= \epsilon z^2-z-1 - Pe^{-z\tau}$ with $\epsilon = c^{-2}$ has at least one negative zero. 
\end{cor}
\begin{proof} Suppose that equation (\ref{NBD}) has an eventually monotone wavefront $u(t,x)=\phi(x+ct)$. Then it follows from the proof of Lemma 25 in \cite{TTT}  that 
$\chi_+(z,\epsilon)$ has a negative zero. 

Conversely, suppose that, given a  fixed pair $(\tau,c)$ of positive parameters,  equation (\ref{NBD}) possesses a semi-wavefront $u(t,x)=\phi(x+ct)$ and  $\chi_+(z,\epsilon)$ has a negative zero $z_1<0$.  We will denote the latter fact as  $(\tau,c)\in \mathcal{D}_m$, where 
$$
\mathcal{D}_m = \left\{\tau \geq 0,  c >0; 
  z^2-cz-1 - Pe^{-zc\tau} =0\ \mbox{has a negative root}\right\}.  
$$
We have to prove that, in such a case, $\phi(t)$ is monotone in some neighborhood of $+\infty$. First, consider 
$P \in (-1,0]$.  Then  $\phi(t)$ is necessarily monotone on $\R$ in view of So and Zou main theorem in \cite{SZ} and the semi-wavefront uniqueness \cite{AGT,TZ,SLW}.  Second, if $P \in (0,1]$, then the monotonicity of  $\phi(t)$ on $\R$ was established in  \cite[Theorem 2.3]{GTLMS}. 

Hence, in what follows,  we can assume that $P>1$.  Then the relation (\ref{ms}) in the appendix says that $(\tau,c)\in  \mathcal{D}_m \subset \mathcal{D}_s$. In consequence, since $1-1/P > (P^2-P)/(P^2+1)$, Proposition \ref{main}  guarantees that $\phi(+\infty)=\ln p$ (i.e. $\phi$ is a wavefront). Moreover, $\phi(t)-\ln p$ decays exponentially to $0$ as $t \to +\infty$  (e.g. see \cite[Lemma 1.1]{GTLMS} or \cite[Proposition 5.6]{TT}). 

Arguing by contradiction, suppose now  that $\phi$ oscillates around $\ln p$  on some connected neighborhood $\mathcal J$ of $+\infty$.  We claim that  these oscillations are necessarily slow on $\mathcal J$.

 


Indeed, if $f$ satisfies the feedback condition (\ref{fco})  and $P >1$ (equivalently, $p \in (e^2,16.999\dots)$ or $\ln p \in (2,2.833\dots)$, cf. \cite[Corollary 2.4]{GTLMS}), then the slow oscillation property of $\phi(t)$ follows  from Proposition \ref{main}.   

Suppose now that $\ln p > 2.833$ and consider the oscillating wavefront $\phi(t)$ of equation (\ref{def}). By Proposition \ref{main}, the first three critical points $t_1< t_2<t_3$ of $\phi(t)$ are finite and  such that $\phi'(t)> 0$ on $ (-\infty, t_1)\cup (t_2,t_3), \ \phi'(t) <0$  on $(t_1,t_2)$ and $\phi(t_1) > \ln p,$ $\phi(t_2) <\ln p < \phi(t_2- \tau)$. We will prove that $\phi(t) >1$ for all 
$t \geq t_1$. First, we observe that, by \cite[Lemma 4.2]{TT} 
\begin{equation}\label{L42}
\phi(t_2) = \xi(t_2-s_2)\left\{\ln p +\frac{1}{\epsilon(\nu -\lambda)} \int_{s_2}^{t_2}\left(e^{\lambda(s_2-u)}- e^{\nu (s_2-u)}\right)f(\phi(u-\tau))du\right\},
\end{equation}
where  $\lambda <0 <\nu $ denote the roots of the equation $\epsilon z^2-z-1 =0$, $s_2$ is the unique point in $ (t_1,t_2)$ where $\phi(s_2) = \ln p$ and    
$\xi(x) = (\nu -\lambda)/(\nu  e^{-\lambda x} -\lambda e^{-\nu  x})$. Since $s_2 \in (t_2-\tau, t_2)$,  $(\tau,c)\in \mathcal{D}_m \subset \mathcal{D}_s$, $P > 1.833$ and $\xi'(x)<0$ for $x>0$, we immediately find that   
$$\phi(t_2) > \xi(\tau)\ln p > \left(1-\frac 1 P\right)(P+1)=P - \frac 1 P \geq 1.833 - 1/1.833>1.$$
Thus $\phi(t) >1$ for all $t \in [t_1,t_3)$ and if $\phi(s) \leq 1$ at some point $s \geq t_3$, then there exists the leftmost critical point $t_*>t_3$ such that 
$\phi(t_*) = \min\{\phi(s), s \in [t_1,t_*]\}  \leq 1,\ \phi'(t_*)=0$ and $\phi''(t_*) \geq 0$.  In particular, $\phi(t_*-\tau) > \phi(t_*)$. Since, in addition, equation (\ref{def}) implies that 
$\phi(t_*) \geq f(\phi(t_*-\tau))$, we obtain that  $\phi(t_*-\tau)> \ln p$. Therefore   $\phi(s_*)=\ln p$ for  some $s_* \in (t_*-\tau, t_*)$.  Invoking again formula (\ref{L42}),  we find that $\phi(t_*) > \xi(\tau)\ln p >1$. This contradiction proves that actually $\phi(t)>1$ for all $t \geq t_1$.  Therefore the feedback condition (\ref{fco}) is satisfied on the 
set $\{\phi(t), \ t \geq t_1\} \subset [1,+\infty)$. Since the slow oscillation property of $\phi(t)$ follows  now from Proposition \ref{main} and Corollary 14 in \cite{TTT},  the claim is proved.

 \vspace{2mm}
 
Finally, arguing as in the second half of the proof of Corollary \ref{28}, we conclude that the assumption $(\tau,c)\in  \mathcal{D}_m$ is not compatible with the existence of slowly oscillating (around $\ln p$) and exponentially converging (toward $\ln p$) wavefront $\phi(t)$. Therefore $\phi(t)$ is monotone at $+\infty$. 
\qed 
\end{proof}

\section{Appendix}
In this section, we are use the  notation
$$
\lambda = \frac{c(c-\sqrt{c^2+4})}{2}, \quad \nu  = \frac{c(c+\sqrt{c^2+4})}{2}, \quad P=\ln p -1, \quad \epsilon = c^{-2}. 
$$ 
The next assertion  follows from  \cite[Theorem 1.1]{TT} and \cite[Theorems 3, 13]{TTT},  is instrumental in proving Theorem \ref{T34} (or Corollary \ref{FC}) of this paper and 
is given for the convenience of the reader. We  consider the following restrictions on  the nonlinearity $f$:
\begin{description}
\item[{\rm \bf(H)}] Let $f \in C^3(\R_+, \R_+)$ have only one critical point
$x_M$ (maximum) and assume that  $f(0)=0$ and $f'(0) >1$.  Suppose further that $0 < f(x) \leq f'(0)x,$ $ x > 0,$ that the equation $f(x)=x$
has exactly two roots $0, \ \kappa>0$ with $\Gamma: = f'(\kappa)<0$, and that the Schwarz
derivative $Sf$ is negative for all $x >0,\  x \not=x_M$:\\
$$(Sf)(x)=f'''(x)(f'(x))^{-1}-(3/2)
\left(f''(x)(f'(x))^{-1}\right)^2 < 0.$$
\end{description}
\begin{proposition} \label{main} Assume {\rm \bf(H)}, suppose that the equation $\epsilon z^2-z-1+f'(0)e^{-z\tau}=0$ has at least one positive root,   and 
\begin{equation}\label{ce} \frac{\nu  - \lambda }{\nu 
e^{-\lambda \tau}- \lambda e^{-\nu  \tau}}
\geq \frac{{\Gamma}^{2} + \Gamma}{{\Gamma}^{2} + 1}.
\end{equation}
Then equation (\ref{def}) has a positive
wavefront $\phi(t)$. Furthermore, if $\phi(t)$ is non-monotone on $\R$,  then  there exist $t_3\geq t_2> t_1,\  t_1 \in \R,\ \ t_2, t_3\in \R \cup \{+\infty\}$, such that $\phi'(t)> 0$ on $ (-\infty, t_1)\cup (t_2,t_3), \ \phi(t_1) > \kappa,$ and $\phi'(t) <0$  on $(t_1,t_2)$. If $t_2$ is finite then $\phi(t_2) <\kappa < \phi(t_2- \tau)$ and $t_3  >t_2$. Finally, if  $f$ 
satisfies the feedback condition 
\begin{equation}\label{fco}
(f(x)-\kappa)(x-\kappa) <0,\ x \in (f(f(x_M)), f(x_M))\setminus\{\kappa\},
\end{equation}
then $\phi(t)$ is either eventually  \mbox{monotone or oscillates  slowly  (see Definition \ref{d2s}) around $\kappa$. }
\end{proposition}

\begin{lem} \label{bb}For each fixed $c>0$ and $P>1$, equation 
 \begin{equation}\label{ckv} 
\Phi(\tau,c):=  \frac{\nu -\lambda}{\nu  e^{-\lambda \tau} - \lambda e^{-\nu  \tau}} =1-\frac 1 P 
\end{equation}
has a unique positive root $\tau=\tau(c)$.  The function $\tau(c):(0,+\infty)\to (0,+\infty)$ is smooth and has a finite limit $\tau(+\infty)= \ln(P/(P-1))=:\hat \tau >0$.\end{lem}
\begin{proof}  It is easy to check that 
$
\partial\Phi(\tau,c)/\partial\tau <0
$
for all $\tau > 0, c >0$. Taking into account that $\Phi(0,c) = 1$, $\Phi(+\infty,c) = 0$ and $P>1$, we deduce the existence of the unique solution $\tau =\tau(c)$
of equation (\ref{ckv}).  Now, let $\hat \tau$ be the  limit of  sequence $\tau(c_j)$ with $c_j \to +\infty$. Then necessarily $e^{-\hat \tau} = 1-1/P$ 
which  proves the uniqueness and finiteness of $\hat\tau >0$. Thus $\tau(+\infty) = \ln(P/(P-1)).$ \qed
\end{proof} 
In the first quadrant of the plane $(\tau,c)$, given a fixed number $P>1$, together with the above defined set $\mathcal{D}_m$ of parameters $(\tau,c)$, we consider also the following domain
$$
\mathcal{D}_s = \left\{\tau \geq 0,  c >0, \ \Phi(\tau,c)=  \frac{\nu -\lambda}{\nu  e^{-\lambda \tau} - \lambda e^{-\nu  \tau}} \geq 1-\frac 1 P\right\}.
$$
Lemma \ref{bb} shows that $\mathcal{D}_s = \{(\tau,c) \in \R_+^2:  0 \leq \tau \leq \tau(c), \ c >0\}$. On the other hand, it was established in \cite[Section 2.3]{GTLMS} that 
$\mathcal{D}_m = \{(\tau,c) \in \R_+^2:  0 \leq \tau \leq T(c), \ c >0\}$, where  $T(c)$ is defined as the unique 
positive solution of equation 
 \begin{equation}\label{tak+} 
\frac{ec^2\tau^2}{2+\sqrt{c^4\tau^2+4c^2\tau^2+4}}\exp\left(\frac{\sqrt{c^4\tau^2+4c^2\tau^2+4}-c^2\tau}{2}\right) = \frac{1}{P}.
\end{equation}
By  \cite[Lemma 1.1]{GTLMS}, $T(c)$ is a smooth decreasing function such that $T(0^+)=+\infty$, $T(+\infty)=T_*$, where 
$PeT_*e^{T_*}=1$.  Observe that $Pe\hat \tau e^{\hat \tau}=eP^2(P-1)^{-1}\ln(P/(P-1))>1$ for $P>1$ so that $\tau(+\infty)=\hat\tau >T_*= T(+\infty)$. 
In fact, the next lemma assures that $T(c) < \tau(c)$ for all $c >0$ so that 
\begin{equation}\label{ms}
\mathcal{D}_m \subset \mathcal{D}_s.  
\end{equation}
\begin{lem} \label{CK} For each $\tau > 0$ and $c>0$, it holds that 
 \begin{equation}\label{ck+} 
\frac{ec^2\tau^2}{2+\sqrt{c^4\tau^2+4c^2\tau^2+4}}\exp\left(\frac{\sqrt{c^4\tau^2+4c^2\tau^2+4}-c^2\tau}{2}\right) > 1 - \frac{\nu -\lambda}{\nu  e^{-\lambda \tau} - \lambda e^{-\nu  \tau}}. 
\end{equation}
\end{lem}
\begin{proof} With  $h =c\tau$, inequality (\ref{ck+}) is equivalent to 
 $$
l=eh^2\exp\left(\frac{2(h^2+1)}{\sqrt{h^2(c^2+4)+4}+ch}\right) > \left(2+\sqrt{h^2(c^2+4)+4}\right)\left(1 - \frac{\nu -\lambda}{\nu  e^{-\lambda \tau} - \lambda e^{-\nu  \tau}}\right)=r.
$$
For fixed $h>0$, we equation maximize $r$  and minimize $l$ with respect to $c \in \R_+$.  The second task is easy: $\frak{l}=\min_{c\geq 0} l=eh^2$. To maximize $r$, we will introduce the new variable $z= (c+\sqrt{c^2+4})/2$, $z \in [1,+\infty)$. Then $r$ takes the form 
$$
r= r(z): =  \left(2+\sqrt{h^2\left(z+\frac 1 z\right)^2+4}\right)\left(1 - \frac{z^2+1}{z^2e^{h/z} +  e^{-hz}}\right).
$$
We will further simplify $r$: 
\begin{equation*}
r(z) <  r_1(z):= \left(4+h\left(z+\frac 1 z\right)\right)\left(1 - \frac{z^2+1}{z^2e^{h/z} +  e^{-hz}}\right).
\end{equation*}
We claim that, for all $z \geq 1$, $r_1(z) \leq eh^2$. 
First, we rewrite this last  inequality in the following equivalents forms: 
$$
1 \leq \frac{t^2+1}{e^{th} +  t^2e^{-h/t}} +\frac{eh^2}{4+h\left(t+\frac 1 t\right)}, \quad t= 1/z \in  (0,1], \  h >0, 
$$
$$
1 \leq \frac{t^2+1}{e^{M t^2} +  t^2e^{-\sigma}} +\frac{e\sigma^2t^2}{4+\sigma\left(t^2+1\right)}, \quad \  t  \in  (0,1], \  \sigma = \frac h t >0, 
$$ 
$$
1 \leq \frac{we^\sigma}{e^{w\sigma} +  w-1} +\frac{e\sigma^2(w-1)}{4+w\sigma}, \quad w= t^2+1 \in  (1,2], \  \sigma >0, 
$$ 
$$
\frac{e^{w\sigma} +  w-1 -we^\sigma}{e^{w\sigma} +  w-1} \leq \frac{e\sigma^2(w-1)}{4+w\sigma}, \quad w  \in  (1,2], \  \sigma >0, 
$$ 
$$
\left(e^{w\sigma} +  w-1 -we^\sigma\right)(4+w\sigma)  \leq e\sigma^2(w-1)(e^{w\sigma} +  w-1), 
$$ 
$$
A(w,\sigma):= e^{w\sigma}( e\sigma^2(w-1)- (4+w\sigma)) + e\sigma^2(w-1)^2+(4+w\sigma)(1 +w(e^\sigma-1)) \geq 0,$$
to be proved for all $w  \in  (1,2], \  \sigma >0. 
$
It is convenient to represent $A(w,\sigma)$ in the form  of the power series:  
$$A(w,\sigma) = w\sum_{k=2}^{+\infty}\frac{A_k(w)}{k!}\sigma^k= w\sigma^2\left(\frac{A_2(w)}{2!}+\frac{A_3(w)}{3!}\sigma+\frac{A_4(w)}{4!}\sigma^2\right)+ w\sum_{k=5}^{+\infty}\frac{A_k(w)}{k!}\sigma^k=$$
$$\frac{w\sigma^2}{4!}\left(12A_2(w)+4{A_3(w)}\sigma+{A_4(w)}\sigma^2\right)+ w\sum_{k=5}^{+\infty}\frac{A_k(w)}{k!}\sigma^k =:B(w,\sigma)+C(w,\sigma),$$ where $A_2(w)= {2}(e-2)(w-1) >0,\ w \in (1,2]$, 
$$
A_3(w)=  (w-1)(6e-7w-4) >0, \quad w \in (1,1.75], 
$$
$$
A_4(w)=  -8w^{3} +12ew^{2} -12ew+4w+4 = -4(w-1)(2w^2+(2-3e)w+1) >0, \quad w \in (1,2], 
$$
$$
A_k(w) = -w^{k-1}(k+4) +ek(k-1)w^{k-2} -ek(k-1)w^{k-3}+4+kw, \quad k \geq 3. 
$$
It follows  from the above expressions that   $A_k(1) =0$  for all $k \geq 2 $ and since for all $k \geq 5$, it holds that
$$
\quad A_k'(1) = -(k-1)(k+4) +ek(k-1)(k-2) - ek(k-1)(k-3)+k = (k-1)(ek-k-4)+k >0, 
$$
as well as for $w  \in  (1,2]$
$$
A_k''(w) =  -w^{k-3}(k+4)(k-1)(k-2) +ek(k-1)(k-2)(k-3)w^{k-4} -ek(k-1)(k-3)(k-4)w^{k-5}$$
$$=w^{k-5}(k-1)\left\{-w^2(k+4)(k-2) +ek(k-2)(k-3)w -ek(k-3)(k-4)\right\} >0, 
$$
we conclude that $A_k(w) >0$ for all $ w  \in  (1,2]$, $k=2,4,5,6,\dots$ and $A_3(w) >0$ for all $ w  \in  (1,1.75]$. This proves that
$A(w,\sigma)  >0$ for all $w  \in  (1,1.75], \  \sigma >0$, and 
 $C(w,\sigma)  >0$ for all $w  \in  (1,2], \  \sigma >0$. 
Hence, to complete the proof of the positivity of $A(w,\sigma)$ for all $\sigma >0$ and $w \in (1,2]$, it suffices  to establish that $B(w,\sigma) >0$ for all  $w  \in  [1.75,2], \  \sigma >0$.
Now,   $B(w,\sigma)$ or, equivalently,  $12A_2(w)+4{A_3(w)}\sigma+{A_4(w)}\sigma^2$ will be positive for all $\sigma>0$ and for some fixed $w \in (1,2]$ if the discriminant 
$$
D(w)=16(A^2_3(w) - 3A_2(w)A_4(w)) = $$
$$
16(w-1)^2\left((6e-7w-4)^2+{24}(e-2)(2w^2+(2-3e)w+1)\right)$$ 
will be negative for this value of $w$.  Observe that $D(w)$ is the product of two quadratic polynomials so that it is immediate to see that $D(w)<0$ for all $w\in (1,2]$. In a consequence,   $B(w,\sigma) >0$ for all  $w  \in  (1,2]$ and $\sigma >0$ and also $A(w, \sigma) > 0$. 

Finally, the positivity of $A(w,\sigma)$ assures that $r =r(z) < r_1(z) \leq eh^2 =\frak{l} \leq l$.
\qed
\end{proof}

\section*{Acknowledgments}  \noindent  The work of Karel Has\'ik, Jana Kopfov\'a and Petra N\'ab\v{e}lkov\'a was supported  by the institutional support
for the development of research organizations I\v CO 47813059.
This work was realized during a stay of Sergei Trofimchuk at the Silesian University in Opava,  Czech Republic. This stay was possible due to the support of  the Silesian University in Opava and the European Union through the project  CZ.02.2.69/0.0/0.0/16\_027/0008521.  S. Trofimchuk  was  also partially  supported by FONDECYT (Chile),   project 1190712. 

\end{document}